\documentclass{amsart}
\usepackage{graphicx}
\usepackage{amsmath}
\usepackage{amsfonts}
\usepackage{amssymb}
\input xypic
\newtheorem{theorem}{Theorem} [section]
\newtheorem{thm}[theorem]{Theorem}

\newtheorem{example}[theorem]{Example}
\newtheorem{lemma}[theorem]{Lemma}

\begin{document}

\title{Milnor Invariants for Spatial Graphs}
\author{Thomas Fleming }

\begin{abstract}

Link homotopy has been an active area of research for knot theorists
since its introduction by Milnor in the 1950s.  We introduce a new
equivalence relation on spatial graphs called component homotopy,
which reduces to link homotopy in the classical case. Unlike
previous attempts at generalizing link homotopy to spatial graphs,
our new relation allows analogues of some standard link homotopy
results and invariants.

In particular we can define a type of Milnor group for a spatial graph under component homotopy,
and this group determines whether or not the spatial graph is splittable.  More surprisingly, we will
also show that whether the spatial graph is splittable up to component homotopy depends only on the
link homotopy class of the links contained within it. Numerical invariants of the relation will also be
produced.
\end{abstract}

\keywords{spatial graph, Milnor numbers, Milnor group, link
homotopy, edge homotopy, component homotopy}

\maketitle

\section{Introduction}

Where knot theory studies embeddings of circles into $S^{3}$,
spatial graph theory studies embeddings of arbitrary graphs.  Let
$G$ be a finite abstract graph.  Let $f$ be an embedding of $G$ into
the $3$-sphere.  We will call $f(G)$ a {\it spatial embedding of
$G$} or simply a {\it spatial graph}.  Since it is the map $f$ that
matters in this construction, we will often refer to $f$ as a
spatial embedding.  We will often use the notation $\Phi := f(G)$ to
emphasize that we are thinking of the spatial graph as a subcomplex
of $S^{3}$.

Several analogues of link homotopy have been proposed for spatial
graphs, such as the edge homotopy and vertex homotopy of Taniyama
\cite{taniyama2}. Two spatial embeddings of a graph $G$ are called
{\it edge-homotopic} if they are transformed into each other by {\it
self crossing changes} and ambient isotopies, where a self crossing
change is a crossing change between two arcs of the same edge, and
{\it vertex-homotopic} if they are transformed into each other by
crossing changes between arcs of two adjacent edges and ambient
isotopies. Each of these notions reduces to link homotopy in the
case when the graph is a disjoint union of circles, and invariants
of these relations have been produced \cite{nikknme, taniyama1}.
Classification of the embeddings of some graphs under these
relations is possible. For example, spatial embeddings of $K_{4}$ up
to edge-homotopy have been completely analyzed by Nikkuni
\cite{nikk1}.

Milnor invariants are sufficient to classify two and three component
links up to link homotopy \cite{m2} and contributed to the eventual
classification of all links \cite{hl}.  So far, no analogue of these
invariants exists for the spatial graph relations listed above.
Since Milnor invariants are a useful tool for studying link
homotopy, our goal is to introduce a generalization of link homotopy
for spatial graphs that allows a reasonable generalization of
Milnor's link homotopy invariants.

The equivalence relation we will use is called \emph{component
homotopy}. Let $SE(G)$ denote the set of all spatial embeddings of
an abstract graph $G$. Two embeddings $\Phi, \Phi' \in SE(G)$ are
called component homotopic if $\Phi$ can be transformed to $\Phi'$
by ambient isotopy and crossing changes between edges that belong to
the same component of $G$.  Note that if $G$ is connected, then any
two embeddings are component homotopic. This is the analogue of the
fact that any knot is link homotopic to the unknot, though for most
graphs there is no canonical choice of embedding to call the
``unknot.'' However, up to component homotopy, there is a clear
choice for the analogue of the unlink, which we will call a
\emph{completely split} embedding.  A spatial embedding $\Phi$ of
$G$ is completely split if each component of $\Phi$ can be separated
from all others by an embedded $S^{2}$.  Note that any two
completely split spatial embeddings $f, g \in SE(G)$ are component
homotopic, and thus up to component homotopy, there is a canonical
``unlink.''

If Milnor's link homotopy invariants vanish for $L$, then $L$ is
link homotopic to the unlink \cite{m1}.  The vanishing of Milnor's
link homotopy invariants is equivalent to the Milnor group of the
link, $ML$, being isomorphic to the free Milnor group $MF$. Choosing
the correct analogue of these groups, which we will call $CM(\Phi)$
and $CMF(G)$, we are able to prove the following theorems, which are
direct generalizations of Theorem 8 and Corollary 2 of Milnor's
original paper \cite{m1}.

\medskip
\noindent\textbf{Theorem \ref{mainthm}} \emph{Let $\Phi$ be an
embedding of a graph $G$ into $S^{3}$.  Then $CMF(G) \cong CM(\Phi)$
where the isomorphism preserves generators (up to conjugacy) if and
only if $\Phi$ is component homotopic to a completely split
embedding.} \medskip

\noindent\textbf{Theorem \ref{itriv}} \emph{Let $\Phi$ be an
embedding of a graph $G$ into $S^{3}$ with components $\Phi_{1}
\ldots \Phi_{n}$.  If the map $\theta$ from $CMF(G)$ to $CM(\Phi)$
sends generators to conjugates of generators, and $\theta^{i}$ is
the induced map on the groups for $\Phi^{i} := \Phi \setminus
\Phi_{i}$, then $ker~\theta \cong ker~ \theta^{i}$ if and only if
$\Phi$ is component homotopic to an embedding where $\Phi_{i}$ can
be separated from the rest of $\Phi$ by an embedded $S^{2}$.}

\medskip

A spatial subgraph of $\Phi$ that is homeomorphic to $\coprod S^{1}$ is
called a \emph{constituent link}.
The following theorem shows that the link homotopy classes of
certain constituent links of $\Phi$ determine whether $\Phi$ is
component homotopic to a completely split embedding. Note that this
is in contrast to edge homotopy and vertex homotopy, where there are
infinite families of examples of non-split embeddings, all of whose
constituent links are link homotopic to trivial links
\cite{nikknme}.

\medskip
\noindent\textbf{Theorem \ref{splitthm}} \emph{Let $L$ be a constituent link of
$\Phi$, where at most one component of $L$ is contained in each component of $\Phi$.  Every such
constituent link of $\Phi$ is link homotopic to the trivial link if and only if $\Phi$ is component
homotopic to a completely split embedding.}
\medskip

It is also possible to extract numerical invariants from the $CM(\Phi)$
by looking at successive nilpotent quotients.  Milnor's numbers can be interpreted as arising from the
elements in the kernel of the map $MF \rightarrow ML$, and so by examining the kernel of the map
$CMF(G) \rightarrow CM(\Phi)$ we will be able to produce similar (though less subtle) numerical
invariants.  Naturally, any invariant of component homotopy is an invariant of edge homotopy and
vertex homotopy, so these numerical invariants are invariants of those relations as well.

We would like to thank Peter Teichner for helpful conversations, and Akira Yasuhara for an
engaging discussion that led to Example \ref{E:akiraeg}.

\section{The Milnor group for spatial graphs}

In this chapter, we will study
embeddings of graphs up to component homotopy.  That is, up to colored
edge homotopy, where every edge in a given component of a graph is labeled
with the same color, and each component of the graph is given a distinct
color.  Notice that for the case $G = \coprod S^{1}$, this is simply link
homotopy.

In an analogue to the colored Milnor group $CMG$ of \cite{ft}, for
an embedding $f: G \rightarrow S^{3}$ we can define a ``free'' group
$CMF(G)$ associated to $G$, and a group $CM(\Phi)$ associated to the
embedding.

The meridians of the edges of $\Phi$ form a normally generating set
for $CM(\Phi):= \pi_{1}(S^{3} \setminus \Phi) /
\ll[m_{ij}^{g_{1}},m_{ik}^{g_{2}}]\gg$, where the generator $m_{ij}$
is the meridian of the $j$th edge in the $i$th component, the
$g_{i}$ are elements of $\pi_{1}(S^{3} \setminus \Phi)$, and $\ll S
\gg$ denotes the subgroup normally generated by the set $S$. In
fact, we need only one generator per loop in $\Phi$, but by
contracting a spanning tree, we may think of these generators as
meridians for a subset of the edges.

Let $CMF(G)$ denote the free colored Milnor group corresponding to
$G$, defined in the following way.  Fix a spanning tree $T$ of $G$.  Label the edges of component
$i$ in $G / T$ by $x_{ij}$.   Define $CMF(G) := F(x_{ij}) / X$, where $F(x_{ij})$ is the free
group with generators $x_{ij}$, and $ X := \ll[x_{ij}^{g_{1}},x_{ik}^{g_{2}}]\gg $, with
$g_{1}, g_{2}$ arbitrary words in the $x_{ij}$.  We will often refer to the relations induced by
$X$ as the \emph{Milnor relations}. When $CMF(G)$ is generated by $x_{ij}$ with $1 \leq i \leq
n$ and $1\leq j \leq r_{i}$ then $CMF(\Phi)$ is generated by meridians $m_{ij}$ with $1 \leq i \leq
n$ and $1\leq j \leq r_{i}$, for the same values of $n$ and $r_{i}$.

For an edge $x_{ij}$ in $G$, there is a unique minimal path $p_{ij}$
in $T$ connecting the end points of $x_{ij}$.  Let $l_{ij} :=
f(x_{ij} \cup p_{ij}) \in CM(\Phi)$.

Note that the free colored Milnor group $CMF(G)$ is a colored
Milnor group in the sense of Freedman and Teichner, and hence is nilpotent by Lemma 3.1 of
\cite{ft}.  As $CM(\Phi)$ is a quotient of $CMF(G)$, it is nilpotent as well.  This is important, as it
allows us to make use of the following theorem of Stallings \cite{stallings}.

\begin{thm}[Stallings] Let $G_{1}$ and $G_{2}$ be nilpotent groups, and
$\psi : G_{1} \rightarrow G_{2}$ a homomorphism.  If the induced map $H_{1}(G_{1})
\rightarrow H_{1}(G_{2})$ is an isomorphism, and the induced map $H_{2}(G_{1})
\rightarrow H_{2}(G_{2})$ is surjective, then $\psi$ is an isomorphism.
\label{stallingsthm}
\end{thm}

With Stallings' theorem at our disposal, it is possible to show that
the group $CM(\Phi)$ is invariant under component homotopy on the
spatial graph $\Phi$ in the same way as for the colored Milnor group
in \cite{ft}. Roughly, the idea is this: given $\Phi_{0}$ a spatial
embedding of $G$, and $h$ be a component homotopy from $\Phi_{0}$ to
$\Phi_{1}$, let $H$ be the track of this homotopy in $S^{3} \times
I$, and $CM(W) := \pi_{1}(S^{3} \times I \setminus H) / X$. Consider
the maps $\psi_{i} : CM(\Phi_{i}) \rightarrow CM(W)$ induced by
inclusion.  The inclusions carry meridians of $\Phi_{i}$ to
meridians of $H$, so the $\psi_{i}$ clearly induce isomorphisms
$H_{1}(CM(\Phi_{i})) \rightarrow H_{1}(CM(W))$. By using finger
moves to introduce self intersections on $H$, we can arrange $CM(W)
\cong \pi_{1}(S^{3} \times I \setminus H)$.  Then by Alexander
duality, $H_{2}(CM(W))$ is generated by the handelbodies dual to
$\Phi_{i}$, and the linking tori at the self intersection points.
These linking tori simply realize the relations in $X$, so
$H_{2}(CM(\Phi_{i})) \rightarrow H_{2}(CM(W))$ is surjective.  Thus,
by Stallings' Theorem, $\psi_{0}$ and $\psi_{1}$ are isomorphisms,
and hence $CM(\Phi_{0}) \cong CM(\Phi_{1})$.

It will be necessary later to understand the structure of
the groups $CM(\Phi)$ in more detail. Given a spatial embedding $\Phi$, a regular neighborhood of
$\Phi_{i}$ is a handlebody, and the 2-cell of this handlebody induces the relation $\prod_{j}
[m_{ij},l_{ij}] =1$ in $\pi_{1}(S^{3} \setminus \Phi)$ and hence in $CM(\Phi)$.  We will call this
relation the \emph{surface relation} induced by $\Phi_{i}$, and the element $\Pi [m_{ij},l_{ij}]$
the \emph{surface element}.

Just as is shown for the classical Milnor group in \cite{m1}, it is
possible to show that $CM(\Phi)$ has the presentation below. The
fact that the only relations in $CM(\Phi)$ are the surface relations
and the Milnor relations will be important in later arguments.

\begin{lemma}The group $CM(\Phi)$ has
 the presentation $CM(\Phi) \cong \{ m_{ij} |r_{1}, \ldots r_{n}, \\
X=1 \}$, where the $m_{ij}$ are meridians to the edges $x_{ij}$,
$r_{i}$ is the surface relation given by $\Phi_{i}$, and $X=1$
denotes the relations arising from $X$.
\label{presentation}
\end{lemma}

\begin{proof}

We begin with a Wirtinger presentation for $\pi_{1}(S^{3} \setminus
\Phi)$, which is calculated from a projection of the spatial graph.
Before projecting, we may first contract a spanning tree of $\Phi$,
as this does not affect $\pi_{1}(S^{3} \setminus \Phi)$.  Given a
projection, choose an arbitrary orientation for each edge, and label
the generators corresponding to each arc of that edge in the diagram
by $m_{ij}^{k}$, $1 \leq k \leq r_{ij}$ in the order the arcs are
met as the edge is traversed. At each vertex we have the relation
that the product of the meridians to all the arcs meeting that
vertex is trivial.  In a fixed component, each edge begins and ends
at the vertex, so we may choose a diagram such that the relation
induced by that vertex is $\prod_{j}
m_{ij}^{1}(m_{ij}^{r_{ij}})^{-1} = 1$.

The relations in the Wirtinger presentation for this diagram are
those coming from the vertices and those of the form $m_{ij}^{k+1} =
w_{ij}^{k}m_{ij}^{k}(w_{ij}^{k})^{-1}$ induced by the crossings,
where the $w_{ij}^{k}$ are arbitrary words in the $m_{ij}^{l}$.
Using the Milnor relations, the relations from the crossings, and a
double induction as in \cite{m1}, we can eliminate the relations
from the crossings and reduce the set of generators to only the
$m_{ij} := m_{ij}^{1}$.  The only remaining relations are those
induced by the vertices, which are of the form $\prod_{j}
m_{ij}(m_{ij}^{w_{ij}})^{-1} = \prod_{j}[m_{ij},w_{ij}] = 1$. These
are precisely the surface relations discussed above.

\end{proof}

Recall that a spatial embedding $\Phi$ of $G$ is completely split if each
component of $\Phi$ can be separated from all others by an embedded
$S^{2}$. We may now formulate an analogue of Theorem 8 of \cite{m1}.

\begin{thm}Let $\Phi$ be an embedding of a graph $G$ into $S^{3}$.
Then $CMF(G) \cong CM(\Phi)$ where the isomorphism preserves
meridinal generators (up to conjugacy) if and only if $\Phi$ is
component homotopic to a completely split embedding. \label{mainthm}
\end{thm}

\begin{proof}

Suppose that $\Phi$ is component homotopic to a completely split embedding.  Since $CM(\Phi)$ is
invariant under component homotopy, we may assume that $\Phi$ is completely split.

The group $CM(\Phi)$ is a free group modulo the set $X$ and the
surface relations generated by the handlebodies forming the
boundaries of neighborhoods of each component of $\Phi$.  The
surface relations in $CM(\Phi)$ are of the form
$\prod[m_{ij},l_{ij}] = 1$ for fixed $i$.  Since each component is
separated from all others, for a fixed $i$, each $l_{ij}$ is a word
in the $m_{ij}$. Since $m_{ij}$ commutes with all other generators
with index $i$, $[m_{ij},l_{ij}] = 1$ for each pair $m_{ij},l_{ij}$.
Thus, in $CM(\Phi)$ the surface relations are trivial, so we have
only the relations introduced by $X$, and the map $x_{ij}
\rightarrow m_{ij}$ is an isomorphism $CMF(G) \rightarrow CM(\Phi)$.


Suppose that an isomorphism $\theta: CMF(G) \rightarrow CM(\Phi)$ exists.  We will induct on the
number of edges in the final component.  Clearly if $\Phi$ has one component, it is component
homotopic to a split embedding, so the base case is trivial.

Suppose $\Phi = \coprod \Phi_{i}$ is a spatial graph of $n$ components, and the first Betti number
$b_{1}(\Phi_{n}) =r$.  Then there are $r$ generators of $CM(\Phi)$ arising from $\Phi_{n}$,
specifically $m_{nj}$, $1 \leq j \leq r$.

If deleting an edge of $\Phi_{n}$ does not change $b_{1}(\Phi_{n})$,
that edge is contractible and so the argument is trivial.  So
suppose that deleting an edge $e$ does change $b_{1}(\Phi_{n})$. We
may choose the spanning tree $T$ to avoid $e$, so let $m_{nr}$ be
the generator of $CM(\Phi)$ represented by a meridian of $e$.

Let $G':= G \setminus e$ and $\Phi'$ be the restriction of $f$ to
$G'$. There is an obvious surjection $CMF(G) \rightarrow CMF(G')$
given by $x_{nr}=1$, and the kernel of this map is clearly $\ll
x_{nr} \gg$.  Let $\theta_{e}$ be the restriction of $\theta$ to
this kernel and $\theta': CMF(G') \rightarrow CM(\Phi)$ be the map
induced by $\theta$.

We want to show that $\theta'$ is an isomorphism.  By the 5 Lemma, this is equivalent to showing
that $\theta_{e}$ is an isomorphism.

$$
\xymatrix{\ll x_{nr} \gg \ar[d]_{\theta_{e}} \ar@{^{(}->}[r] & CMF(G)
\ar[d]_{\theta}^{\cong} \ar@{->>}[r]^{x_{nr}=1} & CMF(G') \ar[d]_{\theta'}
\ar[r] & 1 \ar[d]\\
\ll m_{nr} \gg \ar@{^{(}->}[r] & CMF(\Phi) \ar@{->>}[r]^{m_{nr}=1} &
CM(\Phi') \ar[r] & 1}
$$

Since $\theta$ sends generators to conjugates of generators,
$\theta(x_{nr}) = m_{nr}^{g}$ and so $\theta_{e}(\ll x_{nr} \gg) \subset
\ll m_{nr} \gg$. Clearly, as $\theta_{e}$ is the restriction of an
isomorphism it is injective, so it suffices to show that it is surjective.
Given $m_{nr}^{g_{1}} \in \ll m_{nr} \gg$, we have
$\theta^{-1}(m_{nr}^{g_{1}}) = x_{nr}^{\theta^{-1}(g^{-1}g_{1})} \in \ll
x_{nr} \gg$. Thus $\theta_{e}$ is an isomorphism, and so is $\theta'$.

As $\Phi'$ has fewer edges, and $\theta'$ is an isomorphism carrying
generators to conjugates of generators, we may apply the induction
hypothesis and use component homotopy to convert $\Phi'$ to a
completely split embedding. If we perform this component homotopy on
$\Phi$ simply carrying the extra edge $e$ along, we arrive at the
embedding shown in Figure \ref{setup}.

\begin{figure}[hbtp]

\centering

\begin{picture}(360,144)

\includegraphics{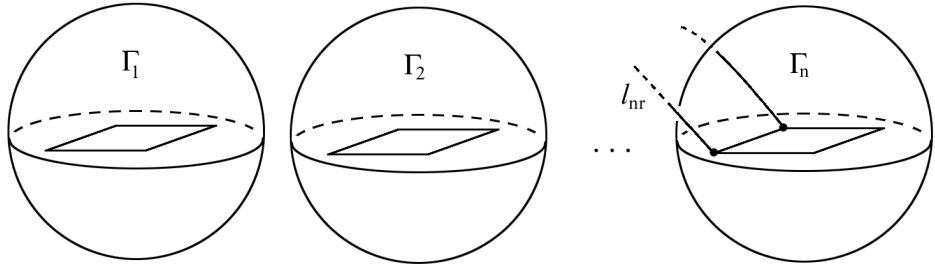}

\end{picture}

\caption{The embedding of $\Phi$ after the inductive hypothesis. The edge labeled $e$ winds
through some or all of the components.}

\label{setup}

\end{figure}

Since the embedding $\Phi'$ is completely split, in $CM(\Phi')$ the
loop $l_{ni}$ is a word in the $m_{nj}$.  Thus, in $CM(\Phi)$,
$l_{ni}$ is a word in $m_{nj}$, $j<r$ and $m_{nr}^{g}$ for any $g
\in \pi_{1}(S^{3} \setminus \Phi)$.
Now, since $[m_{ni}, m_{nj}^{g}] =1$ in $CM(\Phi)$, we know that $[m_{ni},l_{ni}]=1 \in
CM(\Phi)$ for all $i < r$.

We have the relation $\prod_{i\leq r} [m_{ni},l_{ni}] =1$ in $CM(\Phi)$ from the 2-cell of the
neighborhood of $\Phi_{n}$,
but as we just saw that $[m_{ni},l_{ni}]=1$, $i < r$, we may conclude that $[m_{nr},l_{nr}] =1$
in $CM(\Phi)$.

Using the isomorphism $\theta$, we have that $[x_{nr},l_{nr}] =1$ in
$CMF(G)$. In $CMF(G)$, the only elements that commute with a
generator $x_{ij}$ are conjugates of generators of the same color
\cite{ft}.  Thus, we know $l_{nr} = \Pi(x_{nj})^{g_{j}}$ for some
$g_{j} \in CMF(G)$, and hence has the same form in $CM(\Phi)$.

This means that $l_{nr} = \prod m_{nj_{1}}^{g_{nj_{1}}}
[m_{ij_{2}}^{g_{ij_{2}}},m_{ij_{3}}^{g_{ij_{3}}}]$ in $\pi_{1}(S^{3} \setminus \Phi')$.  We
may eliminate the commutators through component homotopy as shown in Figure \ref{linkcancel}.
To eliminate the terms of $m_{nj}^{g_{nj}}$, we use the component homotopy move as in Figure
\ref{graphcancel}.

We may reduce to $l_{nr} =1$ in $\pi_{1}(S^{3} \setminus \Phi')$.
Thus the loop $l_{nr}$ is contractible and may be isotoped into the
sphere containing $\Phi_{n}$, so $\Phi$ is component homotopic to a
completely split embedding.

\end{proof}

\begin{figure}[hbtp]

\centering

\begin{picture}(360,144)

\includegraphics{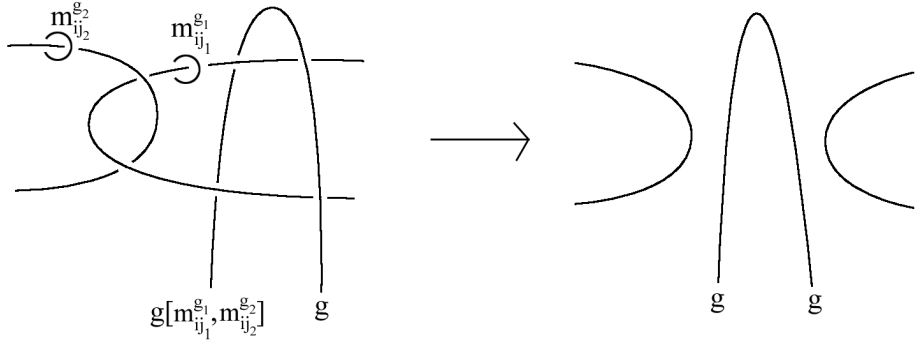}

\end{picture}

\caption{Using component homotopy to cancel elements of the form
$[m_{ij}^{g_{1}},m_{ik}^{g_{2}}]$.}

\label{linkcancel}

\end{figure}

\begin{figure}[hbtp]

\centering

\begin{picture}(360,144)

\includegraphics{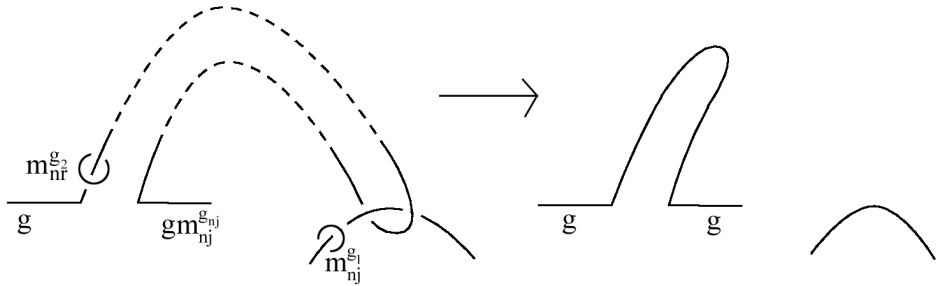}

\end{picture}

\caption{Using component homotopy to cancel elements of the form $m_{nj}^{g_{nj}}$.}

\label{graphcancel}

\end{figure}

Let $CM(\Phi^{i})$ denote the colored Milnor group of $\Phi^{i} := \Phi \setminus \Phi_{i}$, the
spatial graph obtained from $\Phi$ by removing the $i$th component
$\Phi_{i}$.  That is $CM(\Phi^{i}):=CM(\Phi)/(m_{ij}=1)$ for a fixed
$i$.  Given a map $\theta: CMF(G) \rightarrow CM(\Phi)$, we have the sequence

$$
\xymatrix@1{ ker~\theta \ar[r] & CMF(G) \ar[r]^{\theta} & CM(\Phi)}
$$

Let $G^{i} := G \setminus G_{i}$, and let $\theta^{i}$ denote the restriction of the map $\theta$ to
$G^{i}$.  This induces the sequence

$$
\xymatrix@1{ ker~\theta^{i} \ar[r] & CMF(G^{i}) \ar[r]^{\theta^{i}} & CM(\Phi^{i})}
$$

We are now able to produce a theorem analogous to Milnor's Corollary 2 of  \cite{m1}.

\begin{thm}Let $\Phi$ be an embedding of a graph $G$ into $S^{3}$.
Suppose the map $\theta: CMF(G) \rightarrow CM \Phi$ sends generators to
conjugates of generators. Then $ker~\theta^{i} \cong ker~\theta$ if and only if
$\Phi$ is component homotopic to an embedding where $\Phi_{i}$ can be separated from the rest of
$\Phi$ by an embedded $S^{2}$.
\label{itriv}
\end{thm}

\begin{proof}

This proof is similar to that of Theorem \ref{mainthm}. Recall that
the only relations in $CM(\Phi)$ are those induced by $X$ and the
surface relations.  Under $\theta$, the set $X$ in $CM(\Phi)$ is hit
only by elements of $X$ in $CMF(G)$, so $ker~\theta$ is generated by
the preimages of the surface relations.  Suppose that $\Phi$ is
component homotopic to an embedding where $\Phi_{1}$ can be
separated from the other components by an embedded $S^{2}$.
Splitting along this $S^{2}$, Van Kampen's theorem implies that
$\pi_{1}(S^{3} \setminus \Phi) = \pi_{1}(S^{3} \setminus \Phi') *
\pi_{1}(S^{3} \setminus \Phi_{1})$, where $\Phi' := \Phi \setminus
\Phi_{1}$.  As the surface element $\Pi [m_{1j},l_{1j}]$ lies in
$\pi_{1}(S^{3} \setminus \Phi_{1})$ and the surface elements $\Pi
[m_{ij},l_{ij}]$ lie in $\pi_{1}(S^{3} \setminus \Phi')$ for $i \geq
2$, the surface relation given by $\Phi_{1}$ is trivial in
$CM(\Phi)$, and the surface elements given by the $\Phi_{i}$, $i>1$
can be written without the generators $m_{1j}$.  Thus, $ker~\theta
\cong ker~\theta^{i}$ under the map taking $x_{1j}$ to 1 and
$x_{ij}$ to $x_{ij}$ for $i>1$.

\medskip%

Now suppose that the component of interest is $\Phi_{n}$ and
$ker~\theta \cong ker~\theta^{n}$. We now induct on the number of
edges in $\Phi_{n}$ as before. Since $ker~\theta \cong
ker~\theta^{n}$, we have that the surface relation from $\Phi_{n}$
is trivial, that is $\prod_{1}^{r} [x_{nj},l_{nj}] =1$ in $CMF(G)$.
However, by induction we know that $\prod_{1}^{r} [x_{nj},l_{nj}] =
[x_{nr},l_{nr}] = 1$ in $CMF(G)$.  We may now use the same
techniques as in the proof of Theorem \ref{mainthm} to alter $\Phi$
by component homotopy so that $\Phi_{n}$ can be separated from the
other components.

\end{proof}

\section{Numerical invariants and relations to links}

Recall that a \emph{constituent link} is a subgraph of $\Phi$ that is
homeomorphic to a disjoint union of circles.

\begin{thm} Let $L$ be a constituent link of $\Phi$, where at most one component of $L$ is
contained in each component of $\Phi$.  Every such constituent link of $\Phi$ is link homotopic to
the trivial link if and only if $\Phi$ is component homotopic to a completely split embedding.
\label{splitthm}
\end{thm}

\begin{proof}

Let $\Phi$ be a spatial graph component homotopic to a completely
split embedding $\Phi'$, and let $L$ be a constituent sublink of
$\Phi$, with at most one component of $L$ contained in each
component of $\Phi$. The component homotopy from $\Phi$ to $\Phi'$
induces a link homotopy carrying $L$ to $L'$. Each component of
$\Phi'$ can be separated from the others by an embedded $S^{2}$.
Each separating 2-sphere contains a single component of $L'$, which
by additional component homotopy moves (contained in that sphere)
can be unknotted.  Thus, $L$ is link homotopic to the unlink.


Assume that all the relavent constituent links of $\Phi$ are  link
homotopically trivial.  We will work by induction on the number of
components in $\Phi$. For a one component graph, all spatial
embeddings trivially satisfy the condition to be completely split.

Suppose $G = \coprod_{1}^{n} G_{i}$ is an $n$ component graph and
$\Phi := f(G)$ is not split up to component homotopy.  Then by
Theorem \ref{mainthm}, no map from $CMF(G)$ to $CM(\Phi)$ preserving
generators (up to conjugacy) is an isomorphism.  So given the map
$\theta(x_{ij}) = m_{ij}$, there must exist a nontrivial element in
the kernel.  Relabeling the components of $\Phi$ if necessary, it
must be of the form $\prod[x_{1j},l_{1j}]$. So some
$[x_{1j_{1}},l_{1j_{1}}] \neq 1$ in $CMF(G)$.  Delete all edges in
$G_{1}$ except for the edges forming $l_{1j_{1}}$ to produce $G^{1}
:= S^{1} \coprod_{2}^{n} G_{i}$, and form the embedding $\Phi^{1}$
by restricting $f$ to $G^{1}$. We have reduced $\Phi_{1}$ to a
circle, so we may now think of $l_{1j_{1}}$ as the zero framed
pushoff.

Suppose $[x_{1j_{1}},l_{1j_{1}}] = 1$ in $CMF(G^{1})$, then $l_{1j_{1}} =
x_{1j_{1}}^{g}$ in $CMF(G^{1})$ by \cite{ft}.  The map $CMF(G) \rightarrow CMF(G^{1})$
is given by $x_{1j} = 1$ for $j \neq j_{1}$, so the kernel is
normally generated by those $x_{1j}$.  So, in $CMF(G)$, $l_{1j_{1}} = \Pi x_{1j}^{g_{j}}$,
and thus $[x_{1j_{1}},l_{1j_{1}}] = [x_{1j_{1}}, \Pi x_{1j}^{g_{j}}] = 1$ in $CMF(G)$. This
is a contradiction, so $[x_{1j_{1}},l_{1j_{1}}] \neq 1$ in $CMF(G^{1})$, and thus $\theta^{1}:
CMF(G^{1}) \rightarrow CM(\Phi^{1})$ given by the restriction of $\theta$ is not an isomorphism.
So by Theorem \ref{mainthm}, $\Phi^{1}$ is not completely split.

Suppose $[x_{2j},l_{2j}] = 1$ in $CMF(G^{1})$ for all $j$.  Then as discussed in the proof of
Theorem \ref{mainthm} $l_{2j} = \Pi x_{2j}^{g_{j}} \in CMF(G^{1})$, and so $l_{2j} = \prod
m_{2j_{1}}^{g_{2j_{1}}} [m_{ij_{2}}^{g_{ij_{2}}},m_{ij_{3}}^{g_{ij_{3}}}]$ in
$\pi_{1}(S^{3} \setminus \Phi^{1})$.  We may now do component homotopy as in Figure
\ref{linkcancel} and Figure \ref{graphcancel} to separate $\Phi_{2}$ from the rest of $\Phi^{1}$.

Let $G^{1\overline{2}} := G^{1} \setminus G_{2}$, and
$\Phi^{1\overline{2}} := \phi(G^{1\overline{2}})$.  Then, $\Phi^{1}$
is component homotopic to $\Phi^{1\overline{2}} \coprod \Phi_{2}$.
The component homotopy induces a link homotopy on the relevant
constituent links, so by assumption, these constituent links of
$\Phi^{1\overline{2}}$ are link homotopic to trivial links. The
spatial graph $\Phi^{1\overline{2}}$ has $n-1$ components, so by
induction, it is completely split up to component homotopy, and
hence so is $\Phi^{1}$. This is a contradiction.


Thus, there exists some $x_{2j_{2}}$ with $[x_{2j_{2}},l_{2j_{2}}]
\neq 1$ in $CMF(G^{1})$.  Now, as before, delete edges of $G_{2}$ to
reduce it to the circle defined by $l_{2j_{2}}$, forming $G^{12} :=
S^{1} \coprod S^{1} \coprod_{3}^{n}G_{i}$.  Again, restricting $f$,
we obtain an embedding $\Phi^{12}$, and $[m_{2j_{2}},l_{2j_{2}}]
\neq 1$ in $CM(\Phi^{12})$.




Repeat the argument as above until each component of $G$ has been
reduced to a circle, forming $G' := G^{12 \ldots n}$. Form $\Phi'$
by restricting $f$ to $G'$. Notice that $[x_{nj_{n}},l_{nj_{n}}]
\neq 1$ in $CMF(G')$. Now, $\Phi'$ is an $n$ component link $L$
where each component of $L$ came from a distinct component of $G$.
Thus $CM(\Phi')$ is $ML$, and the map $\theta: CMF(G) \rightarrow
CM(\Phi)$ induces a map $\theta': MF \rightarrow ML$, where $MF$ is
the free Milnor group on $n$ generators.

The commutator $[x_{nj_{n}},l_{nj_{n}}] \neq 1$ in $MF$, but is an
element of the kernel of $\theta'$.  Since the map $\theta': MF
\rightarrow ML$ takes generators to conjugates of generators but is
not an isomorphism, by Theorem \ref{mainthm} $L$ is not link
homotopically trivial.

%
%
%
%

\end{proof}

To determine when a spatial graph $\Phi$ is completely split up to
component homotopy, Theorem \ref{splitthm} implies that it is
sufficient to check that the constituent links (of a certain type)
are all split instead of checking the conditions of Theorem
\ref{mainthm}. Note that a similar statement is not true for Theorem
\ref{itriv}. As the following example shows, it is possible for a
every constituent link of $\Phi$ to be $i$-trivial, but $\Phi$ to
not be component homotopic to $\Phi^{i} \coprod \Phi_{i}$.

\begin{example}
\label{E:akiraeg}
{\rm
Let $\Phi$ be the spatial graph shown in Figure \ref{akiraeg}.  Then
clearly every constituent link is $3$-trivial.  However, we have the
relation $[m_{31}, [m_{11},m_{21}]] = 1 \in CM(\Phi)$, but $[x_{31},
[x_{11},x_{21}]] \neq 1 \in CM(G)$.  Thus, for the map $\theta(x_{ij})
= m_{ij}$, we have $ker~\theta \ncong
ker~\theta^{3}$, and so by Theorem \ref{itriv} the embedding $\Phi$
is not
component homotopic to $\Phi^{3} \coprod \Phi_{3}$.
}
\end{example}

\begin{figure}[hbtp]

\centering

\begin{picture}(132,132)

\includegraphics{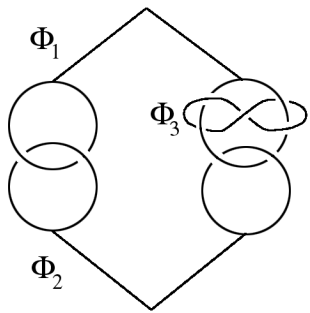}

\end{picture}

\caption{A spatial graph that is not split up to component homotopy.}

\label{akiraeg}

\end{figure}

Teichner and Freedman introduced the notion of colored link homotopy
in \cite{ft}.  In this notion, each component of link is assigned a
color, and crossing changes are allowed between components of the
same color.  Thereom \ref{splitthm} can be easily extended to the
statement that $\Phi$ is component homotopic to a completely split
embedding if and only if every constituent link of $\Phi$ is colored
link homotopic to the unlink, where two components of the link have
the same color if and only if they are contained in the same
component of $\Phi$.  This shows that in some sense, the group
$CM(\Phi)$ contains information about colored link homotopy classes
of all the constituent links.

It is possible to extract numerical invariants from the $CM(\Phi)$ which are the analogue of the
length of the shortest non-vanishing Milnor invariant for a link.

It is well known that the Milnor group of a link is nilpotent, and the group
$CM(\Phi)$ is also nilpotent.  The kernel of the map
$CMF(G) \rightarrow CM(\Phi)$ is generated by surface elements, which are products of
commutators
of elements.  This suggests that examining successive quotients by all
commutators of a fixed length may be productive.

Let $H$ be a subgraph of $G$, and $\Phi := f(G)$ a spatial
embedding. Let $\Phi^{H}$ be the spatial embedding of $G^{H} := G
\setminus H$ that is induced by $f$.  Let $CM(\Phi^{H})$ denote the
colored Milnor group of $\Phi^{H}$, Given a map $\theta: CMF(G)
\rightarrow CM(\Phi)$, we have the sequence below:

$$
\xymatrix@1{ ker~ \theta \ar[r] & CMF(G) \ar[r]^{\theta} & CM(\Phi)}
$$

Let $\theta^{H}$ be the map induced by the map $\theta$.

$$
\xymatrix@1{ ker~ \theta^{H} \ar[r] & CMF(G^{H}) \ar[r]^{\theta^{H}} &
CM(\Phi^{H})}
$$

Then we may define $\lambda_{\Phi}(H)$, a component homotopy
invariant of $\Phi$ as follows:

$$\lambda_{\Phi}(H) + 1 := \mbox{min}~ \{n~|~ ker~ \theta^{H}/[n] \ncong ker~ \theta/[n]\}$$

%
%
%
%

Since the groups $CM(\Phi)$ and $CM(\Phi^{H})$ are invariant under
component homotopy, so is $\lambda_{\Phi}(H)$. Here we consider
$ker~ \theta$ as a subgroup of $CMF(G)$, and $[n]$ denotes all
commutators of length $n$ whose entries are elements of $CMF(G)$.
Notice that if $G = \coprod S^{1}$, then $\lambda_{L}(L_{i})$ is the
length of the index of the shortest nonvanishing Milnor homotopy
invariant whose index contains $i$.

\begin{example}{\rm The first nonvanishing Milnor invariant of the Borromean
rings is $\mu(123)$.  If $L$ is the Borromean rings, then the kernel of the map
$\theta: CMF(L) \rightarrow CML$ is generated by the element $[x_{1},
[x_{2},x_{3}]]$.  When the first component is removed, however, the resulting
link is the unlink, so $ker~ \theta^{L_{1}}$ is trivial.  Thus,
$\lambda_{L}(L_{1}) = 3$. }
\end{example}

\begin{figure}[hbtp]

\centering

\begin{picture}(288,144)

\includegraphics{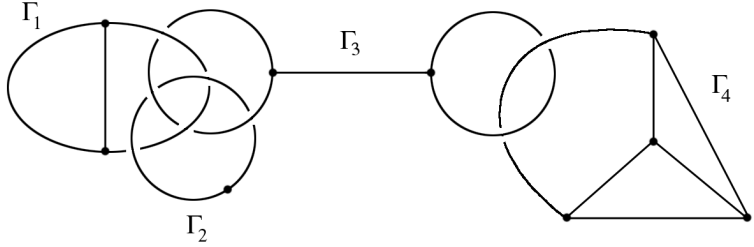}

\end{picture}

\caption{A spatial graph.}

\label{example1}

\end{figure}

\begin{example}
{\rm
Let $\Phi$ be the spatial embedding of $ \Theta_{3} \coprod  S^{1}\coprod H_{2} \coprod K_{4}$
shown in Figure \ref{example1}.  For this graph, $CM(\Phi)$ is isomorphic to $CMF(G)$ modulo
the following additional relations:

$$[m_{13}, [m_{21}, m_{31}]] = [m_{21}, [m_{31}, m_{13}]] = [m_{31}, [m_{13},
m_{21}]][m_{32},m_{41}] = [m_{41}, m_{32}] = 1$$

So, for this spatial graph, $\lambda_{\Phi}(G_{1}) = \lambda_{\Phi}(G_{2}) = 3$ and
$\lambda_{\Phi}(G_{3}) = \lambda_{\Phi}(G_{4}) = 2$.

Notice that while the surface relation given by $\Phi_{3}$ contains
multiple commutators, only the shortest matters for the computation of
$\lambda_{\Phi}(G_{3})$. Also, the presence of the length two commutator
involving $\Phi_{3}$ and $\Phi_{4}$ does not affect the computation of
$\lambda_{\Phi}(G_{1})$.
}
\end{example}

In fact, by the proof of Theorem \ref{splitthm}, when $H = G_{i}$ the value of
$\lambda_{\Phi}(H)$ will be the length of the shortest non-vanishing (colored) Milnor invariant of
any link that contains a cycle from $G_{i}$.

\medskip

\textsc{University of California San Diego, Department of
Mathematics, 9500 Gilman Dr., La Jolla,
 CA 92093-0112}

\emph{E-mail address:} \texttt{tfleming@math.ucsd.edu}


\begin{thebibliography}{99}



\bibitem{nikknme} T. Fleming and R. Nikkuni, \emph{Homotopy on spatial graphs and the
Sato-Levine Invariant}, math.GT/0509003

\bibitem{ft} M. Freedman and P. Teichner, \emph{ $4$-manifold topology I.
Subexponential groups},  Invent. Math.  122  (1995)  no. 3, 509--529

\bibitem{hl} N. Habegger and X. S. Lin, \emph{The classification of links up to link homotopy}, J.
Am. Math. Soc. \textbf{3} (1990) 389-419

\bibitem{m1} J. Milnor, \emph{Link groups}, Ann. Math. \textbf{59} (1954) 177-195

\bibitem{m2} J. Milnor, \emph{Isotopy of links} in: Algebraic and Geometric Topology, A
symposium in honor of S. Lefshetz, ed. R. H. Fox, Princeton
University Press, New Jersey, (1957) 280-306

\bibitem{nikk1} R. Nikkuni, \emph{Edge-homotopy classification of spatial complete graphs on
four vertices}, J. Knot Theory Ramif. \textbf{13} (2004) 763-777

\bibitem{stallings} J. Stallings, \emph{Homology and central series of groups}, J. Algebra,
\textbf{2} (1965) 170-181

\bibitem{taniyama1} K. Taniyama, \emph{Link homotopy invariants of graphs in
$\mathbf{R}^{3}$}, Rev. Mat. Univ. Complut. Madrid \textbf{7} (1994)
129-144

\bibitem{taniyama2} K. Taniyama, \emph{Cobordism, homotopy and homology of graphs in
$\mathbf{R}^{3}$} Topology \textbf{33} (1994) 509-523



\end{thebibliography}
\end{document}